\documentclass[11pt,twoside]{article}

\usepackage{amsthm, amsmath, amssymb, amsfonts, graphicx, epsfig, ulem}
\usepackage{accents}

\usepackage{bbm}
\usepackage{mathrsfs}

\usepackage{epstopdf}

\usepackage{graphicx}
\usepackage[font=sl,labelfont=bf]{caption}
\usepackage{subcaption}

\usepackage{float}
\restylefloat{table}

\usepackage{enumerate}

\usepackage[usenames,dvipsnames]{color}

\usepackage[colorlinks=true, pdfstartview=FitV, linkcolor=blue,
            citecolor=blue, urlcolor=blue]{hyperref}
\usepackage[usenames]{color}

\usepackage{url}
\makeatletter\def\url@leostyle{%
 \@ifundefined{selectfont}{\def\UrlFont{\sf}}{\def\UrlFont{\scriptsize\ttfamily}}} \makeatother

\usepackage[framemethod=default]{mdframed}

\usepackage[margin=1.0in, letterpaper]{geometry}

\newtheorem{theorem}{Theorem}
\newtheorem{proposition}[theorem]{Proposition}
\newtheorem{lemma}[theorem]{Lemma}

\theoremstyle{definition}

\theoremstyle{remark}
\newtheorem{remark}[theorem]{Remark}

\numberwithin{equation}{section}
\numberwithin{theorem}{section}

\def\cA{\mathcal{A}}
\def\cB{\mathcal{B}}
\def\cC{\mathcal{C}}

\def\cE{\mathcal{E}}
\def\cF{\mathcal{F}}

\def\cM{\mathcal{M}}
\def\cN{\mathcal{N}}

\def\cP{\mathcal{P}}
\def\cQ{\mathcal{Q}}

\def\cT{\mathcal{T}}

\def\bE{\mathbb{E}}
\def\bF{\mathbb{F}}

\def\bN{\mathbb{N}}

\def\bP{\mathbb{P}}
\def\bQ{\mathbb{Q}}
\def\bR{\mathbb{R}}
\def\bS{\mathbb{S}}

\def\sF{\mathscr{F}}

\newcommand{\1}{\mathbbm{1}}            
\newcommand{\set}[1]{\{#1\}}            

\floatstyle{plaintop}

\title{ \vspace{-3em} 
    Data-Driven Nonparametric Robust Control under Dependence Uncertainty
}

\author{
        Erhan Bayraktar\thanks{E. Bayraktar is partially supported by the National Science Foundation under grant DMS-2106556 and by the Susan M. Smith chair.
        \newline \hspace*{1.45em}  Department of Mathematics, University of Michigan, Ann Arbor
        \newline \hspace*{1.45em}  530 Church Street, Ann Arbor, MI 48109, USA
        \newline \hspace*{1.45em}  Email: \url{erhan@umich.edu}, URL: \url{https://sites.lsa.umich.edu/erhan/}}
        \and
        \and Tao Chen\,\thanks{
        \newline \hspace*{1.45em}  Email: \url{chenta@umich.edu}, URL: \url{http://taochen.im}
        }}

\date{ {\small 
This Version: \today
}}

\begin{document}

\maketitle

{\footnotesize
\begin{tabular}{l@{} p{350pt}}
  \hline \\[-.2em]
  \textsc{Abstract}: \ & We consider a multi-period stochastic control problem where the multivariate driving stochastic factor of the system has known marginal distributions but uncertain dependence structure. To solve the problem, we propose to implement the nonparametric adaptive robust control framework. We aim to find the optimal control against the worst-case copulae in a sequence of shrinking uncertainty sets which are generated from continuously observing the data.  Then, we use a stochastic gradient descent ascent algorithm to numerically handle the corresponding high dimensional dynamic inf-sup optimization problem. We present the numerical results in the context of utility maximization and show that the controller benefits from knowing more information about the uncertain model.\\[1em]
\textsc{Keywords:} \ & nonparametric adaptive robust control, model uncertainty, stochastic control, adaptive robust dynamic programming, Wasserstein distance, Markovian control problem, utility maximization, copula, machine learning, stochastic gradient descent ascent.
 \\
\textsc{MSC2010:} \ & 49L20, 49J55, 93E20, 93E35, 60G15, 65K05, 90C39, 90C40, 91G10, 91G60, 62G05 \\[1em]
  \hline
\end{tabular}
}


\section{Introduction}
In this paper we propose a nonparametric approach for solving a stochastic Markovian control problem in discrete time under a special type of uncertainty.
We assume that the multivariate driving random factor of the underlying stochastic process has known marginals and unknown dependence.
Naturally, to deal with such kind of uncertainty, one can choose a parametric family of copula functions and learn the unknown parameter from the observed data.
In this work, we avoid making such postulation to prevent model misspecification. When it comes to handling model uncertainty, there are different approaches, parametric and nonparametric, developed in the past decades to incorporate learning into solving control problems with unknown system models (cf. \cite{KV2015}, \cite{CG91}, \cite{Rieder1975}, \cite{CM2020}).
However, earlier studies show that a pure learning approach without awareness of the model risk is prone to risk caused by estimation error and often leads to overly aggressive controls and system outcomes with high variances.
On the other hand, the central idea of robust control goes back to \cite{GS1989}.
A large body of research have been devoted to this area since then, and produced fruitful results.
Robust techniques are extremely successful in dealing with model risk but if the learning phase is lacking in the framework, corresponding controls can be overly conservative and even trivial.
Our work aims to address all the issues mentioned above when handling dependence uncertainty through applying the nonparametric adaptive robust methodology proposed in \cite{BC2022} and developing an efficient numerical scheme for implementing such method.
The main idea of adaptive robust control is to consider a specific type of sequential games: the controller first constructs a sequence of uncertainty sets by continuously observing the system data; then, at each time step the nature chooses one model from the time-$t$ uncertainty set that is the worst for the controller; finally, the controller will apply an optimal control in response, and both players (the controller and the nature) repeat the same process at latter time steps.
Note that such framework is data-driven and the learning phase is integrated into the optimization phase.
Regarding constructing the desired uncertainty sets, we refer to \cite{BCC2017} and \cite{Bhudisaksang2021} for the parametric case in discrete time and continuous time, respectively.
Discussion of constructing time dependent uncertainty sets in discrete time in the nonparametric case can be found in our earlier work \cite{BC2022}.
In summary, to deal with the dependence uncertainty, we use the method described in the aforementioned papers to build uncertainty sets in terms of the Wasserstein ball centered at the nonparametric estimators of the copula by utilizing the concentration results for the empirical distribution and the Wasserstein distance (cf. \cite{BGM1999}, \cite{FG2015}), and formulate the corresponding robust control problem.
Practically, numerical search of the worst-case copula in a Wasserstein ball is extremely difficult.
One can utilize the results presented in \cite{GK17} to obtain a duality result for solving the relevant inf-sup optimization problem.
However, as we will see in the sequel, the Bellman equation derived in this framework has unavoidable high dimension.
Thus, solving such optimization problem using brute force is virtually impossible.
Towards this end, we design a stochastic gradient descent ascent algorithm to mitigate the computational burden.
On top of that, we use the Gaussian process regression model to approximate the corresponding value function.
The gradient of such approximation can be explicitly computed, and this fact increases the numerical efficiency in the dynamic programming backward recursion.

The rest of the paper is organized as follows. We begin Section~\ref{sec:setup} with setting up the model and then discuss the stochastic process that represents the learning of the unknown dependence, and the construction of the uncertainty sets which contain all the considered copulae at different time steps.
Section~\ref{sec:adaptive-robust} is dedicated to the formulation of the nonparametric adaptive robust control framework. We investigate the solution of the nonparametric adaptive robust control problem in Section~\ref{sec:solution}.
In this section, we prove the Bellman principle of optimality for the problem and show the existence of measurable worst-case model selector as well as the existence of measurable optimal control.
We also discuss the comparison among several different methods that can be used to handle the stochastic control problem under dependence uncertainty.
Finally, in Section~\ref{sec:numerics} we consider an illustrative example.
Namely, the uncertain utility maximization problem where the investor needs to allocate the wealth among the money market account and several risky assets without knowing the true dependence of return processes of the risky assets.
We apply the nonparametric adaptive robust control approach to such problem and provide a numerical solver by using machine learning techniques such as stochastic gradient descent ascent and the Gaussian process.
Numerical results presented in this section show that knowing more about the distribution of the underlying stochastic process leads to optimal trading strategy that is more balanced between profit seeking and risk aversion.

\section{Nonparametric Adaptive Robust Control Problem Subject to Dependence Uncertainty}\label{sec:setup}

Let  $(\Omega, \sF)$ be a measurable space,  and $T\in \bN$ be a fixed time horizon. Let $\cT=\set{0,1,2,\ldots,T}$, $\cT'=\set{0,1,2,\ldots,T-1}$, and $\cT''=\set{1,2,\ldots,T}$.
On the space $(\Omega, \sF)$ we consider a controlled stochastic process $X=\{X_t,\ t\in\cT\}$ taking values in $\bR^d$ with dynamics
\begin{align}
  X_{t+1}=S(t,X_t,\varphi_t,Z_{t+1}), \quad t\in\cT'.
\end{align}
The above $Z_t=(Z^{(1)}_t,\ldots,Z^{(n)}_t)$, $t\in\cT''$, is an $n$-dimensional i.i.d. sequence adapted to the natural filtration of the process $X$ which is denoted by $\bF=\{\cF_t, t\in\cT\}$.
The process $\varphi=\set{\varphi_t, t\in\cT'}$ is also $\bF$-adapted and takes values in a compact subset $A\subset\bR^m$.
For every $t\in\cT'$, the function $S(t,\cdot,\cdot,\cdot):\bR^d\times A\times\bR^n\to\bR^d$ is deterministic and continuous.
In addition, we denote by $\cA_t$ the set of all processes that take values in $A$ and are adapted to the filtration $\bF_t:=\{\cF_s, t\leq s\leq T-1\}$.
Each element in $\cA_t$ is called an admissible control starting at time $t$, and we use the convention $\cA=\cA_0$.
In this work, we assume that the process $Z$ is observable, and for every fixed $t\in\cT''$, the marginal distributions of $Z_t$ are known to the controller, but the dependence structure among marginals is unknown and needs to be learned.
We write $\cP(\bR^n)$ as the set of all distributions on $\bR^n$.
Finally, let $\ell:\bR^d\to\bR$ be continuous, bounded from above, and playing the role of the loss function.
In this work, denote by $F_*$ the true distribution of $Z_t$, $t\in\cT''$, we will formulate and solve a robust optimization problem aiming to minimize the expected loss with the restriction that the worst-case distribution of $Z_t$, $t\in\cT''$, matches the marginals of $F_*$.
To proceed, we specify the following notations which will be used throughout.
\begin{itemize}
  \item
  For any $z\in\bR^n$, $z^{(i)}$, $i\in\cN:=\{1,\ldots,n\}$, is the $i$th component of $z$.
  \item
  For any CDF $F$ of a $n$-dimensional random variable, $F^{(i)}$, $i\in\cN$, is the $i$th marginal CDF.
  \item
  For any CDF $F$ of a $n$-dimensional random variable, $F^{-1}=(F^{(1),-1},\ldots,F^{(n),-1})$.
  \item
  For any CDF $F$ of a $n$-dimensional random variable and $z\in\bR^n$,
  \begin{align*}
  F\circ z=&(F^{(1)}(z^{(1)}),\ldots,F^{(n)}(z^{(n)})),\\
  F^{-1}\circ z=&(F^{(1),-1}(z^{(1)}),\ldots,F^{(n),-1}(z^{(n)})).
  \end{align*}
  \item
  For any function $f$ on $\bR^n$, and CDF $F$ of a $n$-dimensional random variable,
  \begin{align*}
    f\circ F=&f(F^{(1)},\ldots,F^{(n)}),\\
    f\circ F^{-1}=&f(F^{(1),-1},\ldots,F^{(n),-1}).
  \end{align*}
\end{itemize}
In order to deal with the dependence uncertainty, we will take the copula perspective and note that
\begin{align}\label{eq:copula}
F_*(z)=C_*\left(F^{(1)}_*\left(z^{(1)}\right),\ldots,F^{(n)}_*\left(z^{(n)}\right)\right),
\end{align}
where $C_*$ is the unknown true copula function of $Z_t$.
Here, we make a standing assumption that $F^{(i)}_*$, $i\in\cN$, is continuous and strictly increasing.
Consequently, the copula function $C_*$ is unique.
In the sequel, we will write $\bP_F$ and $\bE_F$ as the probability measure and expectation, respectively, associated to the distribution function $F$.

In the first step, having in mind the possibility that the parametric family which $C_*$ belongs to could be unknown, we need to define an estimator for $C_*$ based on the observed data of $Z$ in a nonparametric manner.
There are essentially two different ways to achieve this.
One is to use the up-to-time $t$ information of $Z$ to construct a multi-dimensional empirical distribution $\widehat{F}_t$ for $F_*$, and define an estimator of $C_*$ via \eqref{eq:copula}.
The other is to consider a kernel based estimator of $F_*$ which is a smoothed version of $\widehat{F}_t$ instead.
One example in such family of methods is the so-called perturbed empirical distribution.
Both approaches have their advantages: using the empirical distribution will simplify the numerical computation; while the kernel based method offers a better estimator that is more in line with the assumption that $F^{(i)}_*$, $i\in\cN$, are continuous.
To this end, with slight abuse of notations, we choose the avenue of the first method mentioned earlier and denote by $\widehat{F}_t$ the empirical distribution of $F_*$ based on information up to time $t$, where $\widehat{F}_0$ is constructed by observing historical data of $Z$ with sample size $t_0$.
It is not hard to see that $\widehat{F}_t$ satisfies the following updating rule
\begin{align}\label{eq:ecdf}
\widehat{F}_{t+1}(z)=\frac{(t_0+t)\widehat{F}_t(z)+\prod_{i\in\cN}\1_{Z^{(i)}_{t+1}\leq z^{(i)}}}{t_0+t+1},\quad z\in[0,1]^n.
\end{align}
Next, we define
\begin{align}\label{eq:Chat}
\widehat{C}_t(u)=\widehat{F}_t(F^{-1}_*\circ u), \quad u\in[0,1]^n,
\end{align}
and $\widehat{C}_t$ is an estimator of the copula $C_*$.
Moreover, we have
\begin{align}\label{eq:func_r}
\widehat{C}_{t+1}(u)=\frac{(t_0+t)\widehat{C}_t(u)+\prod_{i\in\cN}\1_{[0,u_i]}\left(F_*^{(i)}\left(Z^{(i)}_{t+1}\right)\right)}{t_0+t+1}=:R(t,\widehat{C}_t,Z_{t+1}), \quad u\in[0,1]^n.
\end{align}

\begin{remark}
We want to stress that even though $F_*$ is uncertain, but we assume that the marginals $F^{(i)}_*$, $i\in\cN$, are known, and therefore due to our definition $F^{-1}_*$ is also well-known to the controller.
\end{remark}

In many works, regardless of knowing $F^{(i)}_*$, $i\in\cN$, or not, the empirical copula is defined as $\widehat{C}_t(u)=\widehat{F}_t(\widehat{F}^{-1}_t\circ u)$ which means using the empirical marginals instead of the true ones to construct the estimator for the copula.
Note that in such construction, due to the discontinuity of $\widehat{F}_t$, the resulting $\widehat{C}_t$ is not necessarily a copula.
On the other hand, the estimator of $C_*$ defined in \eqref{eq:Chat} is not a copula either. Hence, any application or computation purely based on either of these two estimators could be problematic.
We argue that, for this reason, one should consider a robust framework to mitigate the inherent risk comes from the statistical estimation procedure.

In rest of the paper, we write
$$
\widehat{F}_{t+1}(z)=\frac{(t_0+t)\widehat{F}_t(z)+H_{t+1}(z-Z_{t+1})}{t_0+t+1},\quad z\in\bR^n,
$$
where $H_t$ is a sequence of continuous distribution function such that $\lim_{t\to\infty}H_t=\1_{[0,\infty)^n}$ if $t$ is allowed to go to infinity, or $H_t\equiv\1_{[0,\infty)^n}$ for any considered $t$.
The former choice corresponds to constructing $\widehat{F}$ as the so-called perturbed empirical distribution and the latter is the case of choosing $\widehat{F}$ to be the regular empirical distribution.
We will prove all the technical results for both choices as our theory works for either one of the two possible frameworks.

To proceed, we fix the notations $\widehat{F}$ and $\widehat{C}$ for estimators defined as in \eqref{eq:ecdf} and \eqref{eq:Chat}, and assume that, for any $0\leq\alpha\leq1$, there exists a deterministic function $r(\alpha,t_0,t)$ decreasing in $t\in\cT'$ such that
\begin{align}\label{eq:concentration}
\bP(d_{W,p}(\widehat{C}_t,C_*)\leq r(\alpha,t_0,t))\geq1-\alpha,
\end{align}
where $d_{W,p}$ is the Wasserstein distance between distributions in $\cP([0,1]^n)$ of order $p\geq1$.
In application, regarding construction of $r(\alpha,t_0,t)$, we will use the concentration results for the empirical distribution in \cite[Theorem 2]{FG2015} by imposing necessary integrability assumptions on $C_*$.
The take-away is that there exists some function $r(\alpha,t_0,t)$ satisfying \eqref{eq:concentration}.

Next, we model the uncertainty of $C_*$, by using \eqref{eq:concentration}, in terms of the confidence region as
\begin{align}
\cC^\alpha_t(\widehat{C}_t)=\left\{C\in\cP([0,1]^n):d_{W,p}(\widehat{C}_t,C)\leq r(\alpha,t_0,t)\right\},
\end{align}
and we call $\cC^\alpha_t$ the uncertainty set of $C_*$.
Since $[0,1]^n$ is a compact subset of $\bR^n$, then $d_{W,p}$ is a metric on $\cP([0,1]^n)$.
Throughout, we consider the metric space $(\cP([0,1]^n),d_{W,p})$.
As mentioned earlier, we will adopt the adaptive robust control framework to our setup in this project.
Hence, we are going to find the worst-case copula function in $\cC^\alpha_t$ and search for the optimal control in reaction to such worst-case copula.
One needs to realize that $\cC^\alpha_t$ includes all types of distributions on $[0,1]^n$ and in general these distributions do not necessarily have uniform marginals.
In other words, not all $C\in\cC^\alpha_t$ are copulae, and this issue will be addressed in the sequel.

Before we discuss and formulate the nonparametric adaptive robust control problem under dependence uncertainty, we first provide the following technical results for preparation.

\begin{lemma}\label{lemma:r-cont}
For fixed $t\in\cT'$, the mapping $R(t,\cdot,\cdot):\cP([0,1]^n)\times\bR\to\cP([0,1]^n)$ is continuous.
\end{lemma}

\begin{proof}
We will prove the result for $H_t$, $t\in\cT'$, being continuous distribution functions.
Regarding the case of $H_t\equiv\1_{[0,\infty)}$, we refer the readers to \cite{CM2020} or Lemma~\ref{lemma:r-cont2} below.

Assume that $(C_i,z_i)\to(C,z)$, where $C_i, C\in\cP([0,1]^n)$, $z_i, z\in\bR$, $i=1,2,\ldots$.
Then, $d_{W,p}(C_i,C)\to0$ and $z_i\to z$.
Denote $\mu_{C_i,z_i}=R(t,C_i,z_i)$ and $\mu_{C,z}=R(t,C,z)$, and $\cM:=\{f:|f(z)-f(z')|\leq\|z-z'\|,\ z, z'\in[0,1]^n\}$.
Note that $[0,1]^n$ is compact, so for any $p\geq1$ there exists a constant $b_p$ such that $d_{W,p}(F,G)\leq b_pd_{W,1}(F,G)$ for any $F, G\in\cP([0,1]^n)$.
Hence, we have that
\begin{align*}
&d_{W,p}(\mu_{C_i,z_i},\mu_{C,z})\leq b_pd_{W,1}(\mu_{C_i,z_i},\mu_{C,z})=b_p\sup\left\{\int_{[0,1]^n}fd\mu_{C_i,z_i}-\int_{[0,1]^n}fd\mu_{C,z}:
 f\in\cM.\right\}
\end{align*}
By the construction of the function $R$, we obtain
\begin{align*}
\sup\Bigg\{\int_{[0,1]^n}&fd(\mu_{C_i,z_i}-\mu_{C,z}):
 f\in\cM.\Bigg\}=\sup\Bigg\{\frac{t_0+t}{t_0+t+1}\int_{[0,1]^n}fd(C_i-C)\\
 &+\frac{\int_{[0,1]^n}fd(H_{t+1}(F^{-1}_*\circ u-z_i)-H_{t+1}(F^{-1}_*\circ u-z))}{t+t_0+1}:\ f\in\cM\Bigg\}.
\end{align*}
Note that
\begin{align*}
H_{t+1}(F^{-1}_*\circ u-z_i)=&\bP(\widetilde{H}\leq F^{-1}_*\circ u-z_i)\\
=&\bP\left(\left(F_*^{(1)}(\widetilde{H}^{(1)}+z_i^{(1)}),\ldots,F_*^{(n)}(\widetilde{H}^{(n)}+z_i^{(n)}\right)\leq u\right),
\end{align*}
where $\widetilde{H}$ is a multi-dimensional random variable with distribution $H_{t+1}$.
Then, we get that
\begin{align*}
&\sup\left\{\int_{[0,1]^n}fd(H_{t+1}(F^{-1}_*\circ u-z_i)-H_{t+1}(F^{-1}_*\circ u-z)):\ f\in\cM\right\}\\
=&\sup\Big\{\bE\left[f\left(F_*^{(1)}(\widetilde{H}^{(1)}+z^{(1)}_i),\ldots,F_*^{(n)}(\widetilde{H}^{(n)}+z^{(n)}_i\right)\right]\\
&\quad\quad-\bE\left[f\left(F_*^{(1)}(\widetilde{H}^{(1)}+z^{(1)}),\ldots,F_*^{(n)}(\widetilde{H}^{(n)}+z^{(n)}\right)\right]:f\in\cM\Big\}\\
\leq&\bE\Big[\Big\|\left(F_*^{(1)}(\widetilde{H}^{(1)}+z^{(1)}_i),\ldots,F_*^{(n)}(\widetilde{H}^{(n)}+z^{(n)}_i)\right)-\left(F_*^{(1)}(\widetilde{H}^{(1)}+z^{(1)}),\ldots,F_*^{(n)}(\widetilde{H}^{(n)}+z^{(n)})\right)\Big\|\Big].
\end{align*}
Since $F^{j}_*$, $j\in\cN$, are all bounded continuous functions, thus by dominated convergence theorem,
$$
\sup\left\{\frac{\int_{[0,1]^n}fd(H_{t+1}(F^{-1}_*\circ u-z_i)-H_{t+1}(F^{-1}_*\circ u-z))}{t+t_0+1}:\ f\in\cM\right\}\to0
$$
as $z_i\to z$.
On the other hand,
\begin{align*}
\sup\Bigg\{\frac{t_0+t}{t_0+t+1}\int_{[0,1]^n}fd(C_i-C):f\in\cM\Bigg\}=\frac{t_0+t}{t_0+t+1}d_{W,1}(C_i,C).
\end{align*}
In view that $d_{W,1}(C_i,C)\leq d_{W,p}(C_i,C)$ for $p\geq1$, we have that $d_{W,p}(C_i,C)\to0$ implies that the above term converges to 0 as well.
Therefore, $d_{W,p}(\mu_{C_i,z_i},\mu_{C,z})\to0$ and the mapping $R(t,\cdot,\cdot)$ is continuous.
\end{proof}

Another important property that $\cC^\alpha_t$ satisfies as a set valued function is upper hemi-continuity (u.h.c.) and it is stated below.

\begin{lemma}\label{lemma:uhc}
For every $t\in\cT'$, the set valued function $\cC^\alpha_t$ is upper hemi-continuous.
\end{lemma}

\begin{proof}
We need to show that for any $C\in\cP([0,1]^n)$ and any open set $E$ such that $\cC^\alpha_t(C)\subset E\subset\cP([0,1]^n)$, there exists a neighborhood $D$ of $C$ such that for all $C'\in D$, $\cC^\alpha_t(C')\subset E$.

To this end, let $\varepsilon=\text{dist}(C,\partial E)-r(\alpha,t_0,t)$ where $\text{dist}(C,\partial E)$ is the shortest Wasserstein-$p$ distance from $C$ to the boundary of $E$.
Since $E$ is an open set, then $\varepsilon>0$.
Hence, we take $D$ as the Wasserstein-$p$ ball centered at $C$ with radius $\varepsilon$.
It is clear that for any $C'\in D$, we have $\cC^\alpha_t(C')\subset E$.
Therefore, for fixed $t\in\cT'$, $\cC^\alpha_t$ is u.h.c..
\end{proof}

\section{Nonparametric Adaptive Robust Control Problem}\label{sec:adaptive-robust}

Now we proceed to formulate the nonparametric adaptive robust control problem assuming the true copula function $C_*$ is unknown.
Recall that we consider $\cP([0,1]^n)$ with the metric $d_{W,p}$.
Obviously, $(\cP([0,1]^n,d_{W,p})$ is separable and complete.
Hence, it is a Polish space and thus a Borel space.
We define the augmented state process $Y=\{Y_t=(X_t,\widehat{C}_t), t\in\cT\}$, and such process has the following dynamics
\begin{align}\label{eq:y-dynamics}
Y_{t+1}=\mathbf{G}(t,X_t,\varphi_t,Z_{t+1}):=(S(t,X_t,\varphi_t,Z_{t+1}),R(t,\widehat{C}_t,Z_{t+1})), \quad t\in\cT'.
\end{align}
According to the assumption that $S$ is continuous and Lemma~\ref{lemma:r-cont}, we get that $\mathbf{G}$ is continuous and therefore Borel measurable.

\subsection{Transition Kernel of the State Process}

Through the rest of this paper, we will write $\cC^\alpha_t(Y_t):=\cC^\alpha_t(\widehat{C}_t)$ for $Y_t=(X_t,\widehat{C}_t)$.
For $C\in\cP([0,1]^n)$, one would like to define the transition kernel of $Y$ as
\begin{align*}
Q_t(dy'|y,a,F)=\bP_{C\circ F^*}(\mathbf{G}(t,y,a,Z_{t+1})\in dy').
\end{align*}
However, $C\circ F^*$ is a well-defined distribution on $\bR^n$ only when $C^{(i)}$, $i\in\cN$, is left continuous at 1 and
$$
\lim_{u\to0}C^{(i)}(u)=0.
$$
Distribution $C\in\cP([0,1]^n)$ does not satisfies the above conditions will result in a distribution $C\circ F^*$ putting mass on infinity.
To overcome this problem, we consider $\overline{\bR}=\bR\bigcup\{-\infty,\infty\}$ with Borel $\sigma$-algebra $\cB(\overline{\bR})=\sigma\left(\cB(\bR)\bigcup\{[-\infty,x),(x,\infty],x\in\bR\}\right)$.
We also consider a metric on $\overline{\bR}$ consistent with $\cB(\overline{\bR})$.
Consequently, we have $\overline{\bR}^n$ equipped with the product topology and it is a Polish space as well as a Borel space.
In addition, we consider $\cP(\overline{\bR}^n)$ the set of all probability distributions on $\overline{\bR}^n$ with topology consistent with weak convergence.
For any $F\in\cP(\bR^n)$, we denote by $\overline{F}$ its extension in $\cP(\overline{\bR}^n)$ as
$$
\overline{F}(x)=
\begin{cases}
F(x), \quad &x\in\bR^n,\\
0, \quad &-\infty\in\{x^{(i)}, i\in\cN\},\\
1, \quad &\infty\in\{x^{(i)}, i\in\cN\}.
\end{cases}
$$
Naturally, we change the state space of $X$ to $\overline{\bR}^n$, and define the augmented state space $E_Y=\overline{\bR}^n\times\cP([0,1]^n)$ with product topology.
It is then a Borel space and the Borel $\sigma$-algebra $\cE_Y$ coincides with the product $\sigma$-algebra.
In addition, for any relevant real valued function $f$, we write
$$
f(\infty)=\lim_{x\to\infty}f(x).
$$

To proceed, for any $t\in\cT'$, $(y,a)\in E_Y\times A$, and distribution $C\in\cP([0,1]^n)$, $Q_t$ is a probability measure on $\cE_Y$ such that
\begin{align*}
Q_t(D|y,a,C)=\bP_{C\circ F_*}(\mathbf{G}(t,y,a,Z_{t+1})\in D), \quad D\in\cE_Y.
\end{align*}
In particular, when $C$ is a copula, $C\circ F_*$ is a $n$-dimensional probability distribution on $\bR^n$ with marginals $F_*^{(1)},\ldots,F_*^{(n)}$.
One important property of $Q_t$ is that it is Borel measurable for fixed $t\in\cT'$.
Such property is crucial for proving the existence of measurable optimal control and worst-case copula.
To this end, we extend Lemma~\ref{lemma:r-cont} to the general case and provide the following technical results.

\begin{lemma}\label{lemma:r-cont2}
For every fixed $t\in\cT'$, the mapping
\begin{align}\label{eq:R2}
\overline{R}(t,F,z')=\frac{(t_0+t)F(z)+\overline{H}_{t+1}(z-z')}{t_0+t+1}, \quad F\in\cP(\overline{\bR}^n),\ z, z'\in\overline{\bR}^n.
\end{align}
is continuous on $\cP(\overline{\bR}^n)\times\overline{\bR}^n$.
\end{lemma}

\begin{proof}
Assume that $(F_i,z'_i)\to(F,z'_0)$ and it implies that $F_i\to F$ in distribution and $z'_i\to z'_0$.
Take any continuity set $B$ of $\overline{R}(t,F_0,z'_0)$.
That is for the boundary $\partial B$ of $B$, the probability $\bP_{\overline{R}(t,F_0,z'_0)}(\partial B)$ is equal to 0.
Therefore, we get that $B$ is also a continuity set of $F$ and $\overline{H}_{t+1}(z-z'_0)$.

Now, since $F_i\to F_0$ in distribution, then by the Portmanteau lemma, we know
$$
\lim_{i\to\infty}\bP_{F_i}(B)\to\bP_{F_0}(B).
$$
On the other hand, for any $z'\in\overline{\bR}^n$, $\overline{H}_{t+1}(z-z')=\bP(\widetilde{H}\leq z-z')=\bP(\widetilde{H}+z'\leq z)$, where $\widetilde{H}$ is the random variable with distribution $\overline{H}_{t+1}$.
Since $z'_i$ converges to the constant $z'_0$, then $\widetilde{H}+z'_i\to\widetilde{H}+z'_0$ in distribution.
In turn, we get
$$
\lim_{i\to\infty}\bP_{\overline{H}_{t+1}(z-z'_i)}(B)\to\bP_{\overline{H}_{t+1}(z-z'_0)}(B).
$$
To summarize, we have that
$$
\lim_{i\to\infty}\bP_{\overline{R}(t,F_i,z'_i)}(B)=\bP_{\overline{R}(t,F_0,z'_0)}(B),
$$
and $\overline{R}(t,F,z')$ is continuous on $\cP(\overline{R}^n)\times\overline{\bR}^n$.
\end{proof}

With this result in hand, we define the transition kernel on the product space $\overline{\bR}^n\times\cP(\overline{\bR}^n)$ with product topology.
That is, fix $t\in\cT'$, for any $(x,F,a,G)\in\overline{\bR}^n\times\cP(\overline{\bR}^n)\times A\times \cP(\overline{\bR}^n)$, and any set $D$ belongs to the Borel $\sigma$-algebra of $\overline{\bR}^n\times\cP(\overline{\bR}^n)$, we write
\begin{align*}
\overline{Q}_t(D|x,F,a,G)=\bP_G((S(t,x,a,Z_{t+1}),\overline{R}(t,F,Z_{t+1}))\in D).
\end{align*}
Then, we show the measurability of $\overline{Q}_t$, $t\in\cT'$.

\begin{lemma}\label{lemma:Qbar}
For every fixed $t\in\cT'$, $\overline{Q}_t$ is a continuous stochastic kernel on $\overline{\bR}^n\times\cP(\overline{\bR}^n)$ given $\overline{\bR}^n\times\cP(\overline{\bR}^n)\times A\times \cP(\overline{\bR}^n)$.
\end{lemma}

\begin{proof}
It is enough to prove that for any bounded continuous function $f$ on $\overline{\bR}^n\times\cP(\overline{\bR}^n)$ we have
\begin{align}\label{eq:bounded_cont}
\int_{\overline{\bR}^n\times\cP(\overline{\bR}^n)}f(y)\overline{Q}_t(dy|x,F,a,G)
\end{align}
is continuous w.r.t. $(x,F,a,G)$.
To this end, we view that
\begin{align*}
\int_{\overline{\bR}^n\times\cP(\overline{\bR}^n)}f(y)\overline{Q}_t(dy|x,F,a,G)=\int_{\overline{\bR}^n}f((S(t,x,a,z),\overline{R}(t,F,z))dG(z).
\end{align*}
By assumption on $S$ and Lemma~\ref{lemma:r-cont2}, for fixed $t\in\cT'$, $S$ and $\overline{R}$ are continuous w.r.t. other variables.
Since $f$ is also continuous, then $f((S(t,x,a,z),\overline{R}(t,F,z))$ is continuous w.r.t. $(x,F,a,z)$.
In addition, $G(z)$ can be seen as a continuous stochastic kernel on $\overline{\bR}^n$ given $(x,F,a,G)$ because it does not depend on $(x,F,a)$.
Hence, by \cite[Proposition 7.30]{BS1978}, we get that
$$
\int_{\overline{\bR}^n}f((S(t,x,a,z),\overline{R}(t,F,z))dG(z)
$$
is continuous w.r.t. $(x,F,a,G)$.
We immediately conclude that \eqref{eq:bounded_cont} is continuous w.r.t. $(x,F,a,G)$.
\end{proof}

Next, we define the following mappings between $\cP([0,1]^n)$ and $\cP(\overline{\bR}^n)$ similarly to \eqref{eq:Chat}: for any $F\in\cP(\overline{\bR}^n)$,
\begin{align}\label{eq:Cbar1}
C(u):=F(F^{-1}_*\circ u), \quad u\in[0,1]^n,
\end{align}
and for any $C\in\cP([0,1]^n)$,
\begin{align}\label{eq:Cbar2}
F(x):=C(F_*\circ x), \quad x\in\overline{\bR}^n.
\end{align}
We realize that due to our assumptions on the marginals $F^{(i)}_*$, $i\in\cN$, the mapping
$$
u\to F_*\circ u, \quad u\in[0,1]^n,
$$
defines a homeomorphism between $[0,1]^n$ and $\overline{\bR}^n$.
By using such observation, we show the main result of this section as follows.

\begin{proposition}\label{prop:bmsk}
For each $t\in\cT'$, the probability $Q_t(\ \cdot\ |y,a,C)$ is a continuous stochastic kernel on $E_Y$ given $E_Y\times A\times \cP([0,1]^n)$.
Moreover, it is a Borel measurable stochastic kernel on $E_Y$ given $E_Y\times A\times \cP([0,1]^n)$.
\end{proposition}

\begin{proof}
Through the proof, all the convergence for probability distributions are understood in the weak sense.
We will prove the statements in several steps.
\begin{enumerate}
\item
Establish a mapping $M_1$ from $E_Y\times A\times\cP([0,1]^n)$ to $\overline{\bR}^n\times\cP(\overline{\bR}^n)\times A\times\cP(\overline{\bR}^n)$ that is continuous.
\item
Use Lemma~\ref{lemma:Qbar} to obtain the continuous mapping $\overline{Q}_t$ from $\overline{\bR}^n\times\cP(\overline{\bR}^n)\times A\times\cP(\overline{\bR}^n)$ to $\cP(\overline{\bR}^n\times\cP(\overline{\bR}^n))$.
\item
Construct a mapping $M_2$ from $\cP(\overline{\bR}^n\times\cP(\overline{\bR}^n))$ to $\cP(E_Y)$, which is restricted to the family of $\overline{Q}_t(\ \cdot\ |x,F,a,G)$, and satisfies some proper form of continuity.
\end{enumerate}
We also need $M_1$ and $M_2$ to satisfy $Q_t=M_2\circ\overline{Q}_t\circ M_1$.
To this end, for the first step, we take $M_1: (x,\widetilde{C},a,C)\mapsto (x,\widetilde{C}\circ F_*,a,C\circ F_*)$, $F, G\in\cP([0,1]^n)$.
Assume that the sequences $\{\widetilde{C}_i\in\cP([0,1]^n),i=1,2,\ldots\}$ and $\{C_i\in\cP([0,1]^n),i=1,2,\ldots\}$ converge to $\widetilde{C}_0\in\cP([0,1]^n)$ and $C_0\in\cP([0,1]^n)$ in distribution, respectively.
It implies that for any continuity point $u$ of $\widetilde{C}_0$, we have $\widetilde{C}_i(u)\to \widetilde{C}_0(u)$ as $i\to\infty$.
In view of the homeomorphism
$$
u\to F_*^{-1}\circ u, \quad u\in[0,1]^n,
$$
if $z$ is a continuity point of $\widetilde{C}_0\circ F_*$, then $\left(F^{(1)}_*(z^{(1)}),\ldots,F^{(n)}_*(z^{(n)})\right)$ is a continuity point of $\widetilde{C}_0$.
Recall that $\widetilde{C}_i$ converges to $\widetilde{C}_0$ in distribution, we get
\begin{align*}
\lim_{i\to\infty}\widetilde{C}_i\circ F_*(z)&=\lim_{i\to\infty}\widetilde{C}_i(F_*\circ z)=\lim_{i\to\infty}\widetilde{C}_i\left(F^{(1)}_*(z^{(1)}),\ldots,F^{(n)}_*(z^{(n)})\right)\\
&=\widetilde{C}_0\left(F^{(1)}_*(z^{(1)}),\ldots,F^{(n)}_*(z^{(n)})\right)=\widetilde{C}_0\circ F_*(z),
\end{align*}
which implies that $\widetilde{C}_i\circ F_*$ converges to $\widetilde{C}_0\circ F_*$ in distribution.
Similarly, $C_i\circ F_*$ converges to $C_0\circ F_*$ in distribution, and step 1 is done.
Also note that step 2 is fulfilled according to Lemma~\ref{lemma:Qbar}.

Next, define $M_2: \overline{Q}_t(\ \cdot\ |x,F,a,G)\mapsto Q_t(\ \cdot\ |x,F\circ F^{-1}_*,a,G\circ F^{-1}_*)$, $F, G\in\cP(\overline{\bR}^n)$.
We will show that $Q_t$ is continuous w.r.t. $(x,F,a,G)$.
In turn, the mapping $M_2$ will be continuous according to the definition of continuous parametric mapping.
Let $D\in\cE_Y$ be a closed set, and let $(F_i,G_i)$ converges to $(F_0,G_0)$.
Similarly to the discussion in step 1, we have $F_i\circ F^{-1}_*$ and $G_i\circ F^{-1}_*$ converge to $F_0\circ F^{-1}_*$ and $G_0\circ F^{-1}_*$, respectively.
Thus, by Lemma~\ref{lemma:r-cont} and the fact that continuity in distribution and continuity in the Wasserstein sense are equivalent on $\cP([0,1]^n)$, the function $\mathbf{G}(t,x,F\circ F^{-1}_*,a,z)$ is continuous w.r.t. $(x,F,a,z)$.
Consider any bounded continuous function $f$ on $E_Y$, we get that
\begin{align}\label{eq:exp2form}
\int_{E_Y}f(y)Q_t(dy|x,F\circ F^{-1}_*,a,G\circ F^{-1}_*)=\int_{\overline{\bR}^n}f(\mathbf{G}(t,x,F\circ F^{-1}_*,a,z))d G(F^{-1}_*\circ z).
\end{align}
The integrand $f(\mathbf{G}(t,x,F\circ F^{-1}_*,a,z))$ is continuous w.r.t. $(x,F,a,z)$.
Also, we view the $G\circ F^{-1}_*$ as a continuous stochastic kernel w.r.t. $(x,F,a,G)$.
By \cite[Proposition 7.30]{BS1978} again, the integral
$$
\int_{\overline{\bR}^n}f(\mathbf{G}(t,x,F\circ F^{-1}_*,a,z))dG(F^{-1}_*\circ z)
$$
is continuous in $(x,F,a,G)$, and so is
$$
\int_{E_Y}f(y)Q_t(dy|x,F\circ F^{-1}_*,a,G\circ F^{-1}_*)
$$
due to \eqref{eq:exp2form}.
As a result, $Q_t$ is continuous in $(x,F,a,G)$, and step 3 is complete.

Finally, we note that $Q_t=M_2\circ\overline{Q}_t\circ M_1$.
According to the above discussion and equivalence between continuity in distribution and in the Wasserstein sense on $\cP([0,1]^n)$, $Q_t$ is a continuous stochastic kernel.
Hence, it is also a Borel measurable stochastic kernel.
\end{proof}

\subsection{Formulation of Adaptive Robust Control Problem}

In this work, we will formulate and solve a closed loop feedback control problem.
To this end, with slight abuse of notations, we say that a control process $\varphi$ is Markovian if for every $t\in\cT'$
\begin{align*}
\varphi_t=\varphi_t(Y(t)),
\end{align*} 
such that $\varphi_t(\cdot):E_Y\to A$ is a measurable mapping.
Similarly, a process $\psi$ is called a Markovian model selector if
\begin{align*}
\psi_t=\psi_t(Y(t)),
\end{align*}
where $\psi_t(\cdot):E_Y\to\cP([0,1]^n)$ is measurable.
In the adaptive robust framework, we consider the Markovian control processes and Markovian model selectors such that $\psi_t(y)\in\cC_t^{\alpha}(y)$ and $\psi_t(y)$ is restricted to be a copula for any $y\in E_Y$.
For every $t\in\cT'$, any time-$t$ state $y_t\in E_Y$, and control process $\varphi\in\cA_t$, we define
\begin{align*}
\mathbf{\Psi}^{y_t}_{t}=\left\{\psi:\psi_s(y_s)\in\cC^\alpha_s(y_s), \ t\leq s\leq T-1\right\}.
\end{align*}
In addition, let $\varphi\in\cA_t$, and $\delta_y$ be the Dirac probability measure assigns the mass at $y\in E_Y$.
Then, at time $t$, for the current state $Y_t=y$, define the probability measure $\bQ^{\varphi,\psi}_{y,t}$ on the concatenated canonical space $\textsf{X}_{s=t}^T E_Y$ as
\begin{align*}
\bQ^{\varphi,\psi}_{y,t}(B_{t}\times\cdots\times B_T)=\int_{B_{t}}\int_{B_{t+1}}\cdots\int_{B_T}\prod_{s=t}^{T-1}Q_s(dy_{s+1}|y_s,\varphi(s,y_s),\psi(s,y_s))\delta_{y}(dy_t).
\end{align*}
Accordingly, we consider the family of probability measures $\cQ^{\varphi}_{y,t}=\left\{\bQ^{\varphi,\psi}_{y,t},\ \psi\in\mathbf{\Psi}_{t}^y\right\}$, and we write, for simplicity, $\cQ^\varphi_y=\cQ^\varphi_{y,0}$.
Finally, for $Y_0=y\in E_Y$, the nonparametric adaptive robust control problem in this work is formulated as
\begin{align}\label{eq:ar1}
\inf_{\varphi\in\cA}\sup_{\bQ^\psi\in\cQ^\varphi_{y}}&\bE_{\bQ^\psi}[\ell(X_T)],\\
\text{s.t.} \quad &\psi_t \text{ is a copula, } t\in\cT'.\nonumber
\end{align}
Note that for every $t\in\cT'$, any $y\in E_Y$, search for the worst-case copula in $\cC^\alpha_t(y)$ is practically impossible.
To overcome such difficulty, we adopt the method of penalty functions and use it to reformulate \eqref{eq:ar1}.
Ultimately, for every fixed $t\in\cT'$ and any fixed $(y,a)\in E_Y\times A$, we want to find the optimizer $\psi_t(y,a)$ such that $\psi_t(y,a)\circ F_*$ has marginals $F_*^{(i)}$, $i\in\cN$.
Hence, we denote by $C_b(\overline{\bR})$ the set of all continuous, bounded, real valued functions, and consider the penalty function
\begin{align*}
\rho(\psi)=\sum_{t\in\cT'}\rho_t(\psi_t):=\sum_{t\in\cT'}\sum_{i\in\cN}\sup_{f_i\in C_b(\overline{\bR})}\left(\bE_{(\psi_t\circ F_*)^{(i)}}[f_i]-\bE_{F^{(i)}_*}[f_i]\right)
\end{align*}
The choice of the penalty function is inspired by the criterion that characterize the weak convergence of probability distributions, and it is similar to the penalty function used in \cite{GK17}.
Note that for such penalty, we have
\begin{lemma}\label{lemma:penalty}
For any fixed $t\in\cT'$, $\psi_t$ is a copula if and only if $\rho_t(\psi_t)=0$.
\end{lemma}
\begin{proof}
  If $\psi_t$ is a copula, then it is obvious that $\rho_t(\psi_t)=0$.
  On the other hand, note that for any $i\in\cN$,
  $$
  \sup_{f\in C_b(\overline{\bR})}\left(\bE_{(\psi_t\circ F_*)^{(i)}}[f]-\bE_{F_*^{(i)}}[f]\right)\geq\bE_{(\psi_t\circ F_*)^{(i)}}[0]-\bE_{F_*^{(i)}}[0]=0.
  $$
  Hence, if $\rho_t(\psi_t)=0$, we obtain that for any $i\in\cN$
  \begin{align*}
    \sup_{f\in\cC_b(\overline{\bR})}\left(\bE_{(\psi_t\circ F_*)^{(i)}}[f]-\bE_{F_*^{(i)}}[f]\right)=0.
  \end{align*}
  Moreover, we get that for any $i\in\cN$, and any $f\in C_b(\overline{\bR})$
  \begin{align}\label{eq:0penalty}
    \bE_{(\psi_t\circ F_*)^{(i)}}[f]-\bE_{F_*^{(i)}}[f]=0.
  \end{align}
  Otherwise, there exists some $i\in\cN$ and $f_0\in C_b(\overline{\bR})$ such that
  $$
  \bE_{(\psi_t\circ F_*)^{(i)}}[f_0]-\bE_{F_*^{(i)}}[f_0]<0.
  $$
  In turn, for such $i$, we have
\begin{align*}
\sup_{f\in C_b(\overline{\bR})}\left(\bE_{(\psi_t\circ F_*)^{(i)}}[f]-\bE_{F_*^{(i)}}[f]\right)\geq&\bE_{(\psi_t\circ F_*)^{(i)}}[-f_0]-\bE_{F_*^{(i)}}[-f_0]\\
=&-\left(\bE_{(\psi_t\circ F_*)^{(i)}}[f_0]-\bE_{F_*^{(i)}}[f_0]\right)>0,
\end{align*}
and $\rho_t(\psi_t)>0$ as a result which is a contradiction.
Therefore, equality~\ref{eq:0penalty} holds true.
Consequently, by the criterion of weak convergence, we get that $\psi_t$ is a copula.
\end{proof}
The above result justifies the choice of $\rho$.
Another advantage of using the proposed penalty function is that for any fixed $f\in C_b(\overline{\bR})$, the linearity in $\psi_t\circ F_*$ allows us to obtain a tractable numerical computation scheme (cf. Section~\ref{sec:algorithm}) for the value functions to be defined later.
To proceed, we provide the following technical result on the regularity property of $\rho_t$, $t\in\cT'$, which plays an important role in proving the existence of optimal control.
\begin{lemma}\label{lemma:lsc_pena}
For every $t\in\cT'$, the penalty term $\rho_t$ is lower semi-continuous (l.s.c.).
\end{lemma}
\begin{proof}
We will show that $\liminf_{C\to C_0}\rho_t(C)\geq\rho_t(C_0)$ for any $C_0\in\cP([0,1]^n)$.
To this end, let $C_j$ be any sequence that converges to $C_0$.
If $\rho_t(C_0)=+\infty$, then, for any $b\in\bR$, there exists a family $\{f_i,i\in\cN\}\subset C_b(\overline{\bR})$ such that
$$
\sum_{i\in\cN}\left(\bE_{(C_0\circ F_*)^{(i)}}[f_i]-\bE_{F_*^{(i)}}[f_i]\right)>b.
$$
Since $C_j\to C_0$, then there exists a $n_0>0$ such that for any $j\geq n_0$ we get
$$
\sum_{i\in\cN}\left(\bE_{(C_j\circ F_*)^{(i)}}[f_i]-\bE_{F_*^{(i)}}[f_i]\right)>\frac{b}{2},
$$
where continuity of $\bE_{(C\circ F_*)^{(i)}}[f_i]$ in $C$ is obtained according to the proof of Proposition~\ref{prop:bmsk}.
Thus, we have that
$$
\rho_t(C_j)>\frac{b}{2}, \quad j\geq n_0,
$$
and moreover
$$
\lim_{j\to\infty}\rho_t(C_j)=+\infty.
$$
If $\rho_t(C_0)<+\infty$, then for any $\varepsilon>0$, there exists a family $\{f_i,i\in\cN\}\subset C_b(\overline{\bR})$ such that
$$
\sum_{i\in\cN}\left(\bE_{(C_0\circ F_*)^{(i)}}[f_i]-\bE_{F_*^{(i)}}[f_i]\right)>\rho_t(C_0)-\varepsilon.
$$
Hence, from the convergence of $C_j$ to $C_0$, there exists $n_0>0$ and for any $j\geq n_0$ the following inequality holds true
\begin{align*}
\sum_{i\in\cN}\left(\bE_{(C_j\circ F_*)^{(i)}}[f_i]-\bE_{F_*^{(i)}}[f_i]\right)>\rho_t(C_0)-2\varepsilon.
\end{align*}
Thus, $\rho_t(C_j)\geq\rho_t(C_0)-2\varepsilon$ for any $j\geq n_0$, and $\lim_{j\to\infty}\rho_t(C_j)\geq\rho_t(C_0)-2\varepsilon$.
Because $\varepsilon$ is arbitrary, we get that $\lim_{j\to\infty}\rho_t(C_j)\geq\rho_t(C_0)$.
In summary, we conclude that $\liminf_{C\to C_0}\rho_t(C)\geq\rho_t(C_0)$ and $\rho_t$ is l.s.c..
\end{proof}
In the sequel, we will consider a reformulation of \eqref{eq:ar1} as follows
\begin{align}\label{eq:ar2}
\inf_{\varphi\in\cA}\sup_{\bQ^\psi\in\cQ^\varphi_{y}}\left(\bE_{\bQ^\psi}[\ell(X_T)]-\rho(\psi)\right).
\end{align}

Before discussing the solution of \eqref{eq:ar2}, we want to remark that, in this work, the uncertainty set is some set of measures on $\cP([0,1]^n)$ centered at the empirical copula.
Another approach is to consider the worst-case model of $Z_t$, $t\in\cT''$, chosen from some uncertainty set centered around the empirical distribution of $Z_t$.
By imposing a penalty function similar to $\rho$, we can ensure that the worst-case model matches $F_*$ in terms of the marginals.
Thus, we are able to set up a problem that is similar to \eqref{eq:ar2} without considering the copula.
Regarding the comparison between such approach and our method, we refer to Section~\ref{sec:comparison} below.
In case that the marginals are also uncertain, one can extend the framework in \cite{BC2022} to the multi-dimensional case via the multivariate empirical distribution.

\section{Solution of the Nonparametric Adaptive Robust Control Problem}\label{sec:solution}

In this section, we show that the solution of the nonparametric adaptive robust control problem \eqref{eq:ar2} is given by solving the following adaptive robust Bellman equations

\begin{align}\label{eq:bellman1}
V_T(y)&=\ell(x), \quad y\in E_Y,\nonumber \\
V_t(y)&=\inf_{a\in A}\sup_{C\in\cC^\alpha_t(y)}\Big(\bE_{C\circ F_*}[V_{t+1}(\mathbf{G}(t,y,a,Z_{t+1})]-\rho_t(C)\Big), \quad y\in E_Y,\ t\in\cT'.
\end{align}

\subsection{Existence of Measurable Optimal Control}

To show existence of measurable optimal control to the problem \eqref{eq:bellman1}, we denote
\begin{align*}
v_t(y,a,C)&=\bE_{C\circ F_*}[V_{t+1}(\mathbf{G}(t,y,a,Z_{t+1})]-\rho_t(C), \quad t\in\cT',
\end{align*}
and provide the following preliminary results.

\begin{proposition}\label{prop:usc}
For every $t\in\cT$, the functions $v_t$ and $V_t$ are upper semi-continuous (u.s.c.).
Moreover, there exists a Borel measurable function $\psi^*_t$ takes values in the set of copulae such that for any fixed $(y,a)$
\begin{align}\label{eq:worst_model}
\sup_{C\in\cC^\alpha_t(y)}v_t(y,a,C)=v_t(y,a,\psi^*_t(y,a)).
\end{align}
\end{proposition}

\begin{proof}
By our assumption, the function $V_T(y)=\ell(x)$ is continuous and hence u.s.c..
For any $C\in\cP([0,1]^n)$, we write
\begin{align*}
\bE_{C\circ F_*}[V_T(\mathbf{G}(T-1,y,a,Z_{T}))]=\int_{E_Y}V_T(y_T)dQ_{T-1}(y_T|y,a,C).
\end{align*}
According to Proposition~\ref{prop:bmsk}, $Q_{T-1}$ is a continuous stochastic kernel.
Then, in view of the assumption that $\ell$ is bounded from above and \cite[Proposition 7.31]{BS1978}, we get that
$$
\bE_{C\circ F_*}[V_T(\mathbf{G}(T-1,y,a,Z_{T}))]
$$
is u.s.c..
On the other hand, by Lemma~\ref{lemma:lsc_pena}, $-\rho_{T-1}(C)$ is u.s.c., and in turn, $v_{T-1}$ are u.s.c..

Next, since that $[0,1]^n$ is compact, then $\cP([0,1]^n)$ (equipped with metric $d_{W,p}$) is compact.
Let $\mathrm{D}=\bigcup_{(y,a)\in E_Y\times A}(\{(y,a)\}\times\cC^\alpha_t(y))$ which can be viewed as the graph of the set valued function $\cC^\alpha_t(y)$.
Note that $\cC^\alpha_t(y)$ is closed and, by Lemma~\ref{lemma:uhc}, is u.h.c..
Hence, the set $\mathrm{D}$ is closed.
According to \cite[Proposition 7.33]{BS1978}, the function $\check{V}_{T-1}(y,a)=\sup_{C\in\cC^\alpha_t(y)}v_{T-1}(y,a,C)$ is u.s.c..
As a result, $V_{T-1}(y)=\inf_{a\in A}\check{V}_{T-1}(y,a)$ is u.s.c. as well.
Moreover, for any fixed $(y,a)$, there exists a Borel measurable function $\psi^*$ such that \eqref{eq:worst_model} holds true.

Last but not the least, for any $C\in\cC^\alpha_{T-1}(y)$ that is not a copula, it is not hard to see that $\rho_{T-1}(C)=+\infty$.
Consequently, for such $C$, $v_{T-1}(y,a,C)=-\infty$.
For any $C$ that is a copula, we have
$$
v_{T-1}(y,a,C)=\bE_{C\circ F_*}[V_T(\mathbf{G}(t,y,a,Z_T))]\geq-\infty.
$$
This indicates that for any $(y,a)$, $\psi^*_{T-1}(y,a)$ is a copula. 
The rest of the proof follows analogously.
\end{proof}

The above proposition proves the existence of measurable worst-case copula selector.
Regarding the optimal control, we use the semi-analyticity property of the value function $V_t$ as follows.

\begin{proposition}\label{prop:lsa}
For every $t\in\cT$, the function $V_t$ is lower semi-analytic (l.s.a.).
Moreover, for any $\varepsilon>0$, there exists an analytically measurable function $\varphi^{*,\varepsilon}_t$ such that
\begin{align}\label{eq:optimal_stratgy}
\sup_{C\in\cC^\alpha_t(y)}v_t(y,\varphi^{*,\varepsilon}_t(y),C)=
\begin{cases}
V_t(y)+\varepsilon \quad &\text{if } V_t(y)>-\infty,\\
-1/\varepsilon \quad &\text{if } V_t(y)=-\infty.
\end{cases}
\end{align}
In other words, $\varphi^{*,\varepsilon}_t$ is the $\varepsilon$-optimal control.
\end{proposition}

\begin{proof}
Since $V_T(y)=\ell(x)$ is continuous by assumption, then $V_T$ is l.s.a..
Recall
\begin{align*}
v_{T-1}(y,a,C,\lambda_{T-1})=\int_{E_Y}V_T(y_T)Q_{T-1}(dy_T|y,a,C)-\rho_{T-1}(C).
\end{align*}
The first term on the RHS is l.s.a. as $Q_{T-1}$ is a Borel measurable stochastic kernel by Proposition~\ref{prop:bmsk}, and the second term $-\rho_{T-1}(C)$ is l.s.a. because it is u.s.c..
Therefore, the function $v_{T-1}(y,a,C)$ is l.s.a..

In view of \eqref{eq:worst_model} in Proposition~\ref{prop:usc} where $\psi^*_{T-1}$ is Borel measurable, we deduce that
$$
\widehat{V}_t(y,a)=\sup_{C\in\cC^{\alpha}_{T-1}(y)}v_{T-1}(y,a,C)
$$
is l.s.a. according to part (3) of \cite[Lemma 7.30]{BS1978}.

Next, the set $A$ is also Borel measurable and in turn analytic.
Hence, by \cite[Proposition 7.50]{BS1978}, $V_{T-1}(y)$ is l.s.a., and \eqref{eq:optimal_stratgy} holds true for $t=T-1$.
The rest of the proof follows analogously.
\end{proof}

\begin{remark}
We want to stress that in general the result in Lemma~\ref{prop:lsa} cannot be improved in the sense that one is unable to obtain the measurable optimal control instead of $\varepsilon$-optimal control.
Following \cite[Proposition 7.50]{BS1978}, on the set
$$
I=\left\{y\in E_Y: \exists a_y\in A \text{ such that } \widehat{V}_t(y,a_y)=V_t(y)\right\}
$$
one will get a universally measurable optimal control.
In our setup, all the continuous points of $V_t$ belong to $I$.
However, due to the presence of the penalty term, it is impossible to show the lower semi-continuity of $V_t$ to the best of our knowledge.
\end{remark}

Now, we proceed to show that the problem \eqref{eq:ar2} satisfies the Bellman principle and is solved by \eqref{eq:bellman1}.

\subsection{Dynamic Programming}

Here we will use Proposition~\ref{prop:usc} and Proposition~\ref{prop:lsa} to show that problem \eqref{eq:ar2} is solved by the Bellman equation \eqref{eq:bellman1}.
Denote
$$
\rho_{t:T-1}(\psi)=\sum_{s=t}^{T-1}\rho_s(\psi_s),
$$
then we have the main result of this section as follows.

\begin{theorem}
For every $t\in\cT'$, and any $y\in E_Y$, we have
\begin{align*}
V_t(y)=\inf_{\varphi\in\cA_t}\sup_{\bQ\in\cQ^\varphi_{y,t}}\left(\bE_{\bQ}[\ell(X_T)]-\rho_{t:T-1}(\psi)\right).
\end{align*}
\end{theorem}

\begin{proof}
We will prove the statement via backward induction in $t=T-1,\ldots,1,0$.

When $t=T-1$ and $y\in E_Y$, we have
\begin{align*}
&\inf_{\varphi\in\cA_{T-1}}\sup_{\bQ^\psi\in\cQ_{y,T-1}^\varphi}\left(\bE_{\bQ^\psi}[\ell(X_T)]-\rho_{T-1}(\psi)\right)\\
=&\inf_{a\in A}\sup_{C\in\cC^\alpha_{T-1}(y)}\left(\int_{E_Y}V_T(y')Q_{T-1}(y'|y,a,C)-\rho_{T-1}(C)\right)=V_{T-1}(y).
\end{align*}
Next, for $t=T-2,\ldots,0$ and $y\in E_Y$, denote $\varphi^t=\{\varphi_s,s=t,\ldots,T-1\}$ and $\psi^t=\{\psi_s,s=t,\ldots,T-1\}$, then
\begin{align*}
&\inf_{\varphi\in\cA_t}\sup_{\bQ^\psi\in\cQ_{y,t}^\varphi}\left(\bE_{\bQ^\psi}[\ell(X_T)]-\rho_{t:T-1}(\psi)\right)\\
=&\inf_{(a,\varphi^{t+1})\in\cA_t}\sup_{C\in\cC^\alpha_t(y)}\int_{E_Y}\sup_{\bQ^{\psi^{t+1}}\in\cQ^{\varphi^{t+1}}_{y',t+1}}\left(\bE_{\bQ^{\psi^{t+1}}}[\ell(X_T)]-\rho_{t:T-1}((C,\psi^{t+1}))\right)Q_t(dy'|y,a,C),
\end{align*}
where by induction
\begin{align*}
&\sup_{\bQ^{\psi^{t+1}}\in\cQ^{\varphi^{t+1}}_{y',t+1}}\left(\bE_{\bQ^{\psi^{t+1}}}[\ell(X_T)]-\rho_{t:T-1}((C,\psi^{t+1}))\right)\\
=&\sup_{\bQ^{\psi^{t+1}}\in\cQ^{\varphi^{t+1}}_{y',t+1}}\left(\bE_{\bQ^{\psi^{t+1}}}[\ell(X_T)]-\rho_{t+1:T-1}(\psi^{t+1})-\lambda_t\rho_t(C)\right)\\
\geq&V_{t+1}(y')-\rho_t(C).
\end{align*}
Hence, we get
\begin{align}\label{eq:loe}
&\inf_{\varphi\in\cA_t}\sup_{\bQ^\psi\in\cQ_{y,t}^\varphi}\left(\bE_{\bQ^\psi}[\ell(X_T)]-\rho_{t:T-1}(\psi)\right) \nonumber\\
\geq&\inf_{(a,\varphi^{t+1})\in\cA_t}\sup_{C\in\cC^\alpha_t(y)}\int_{E_Y}\left(V_{t+1}(y')-\rho_t(C)\right)Q_t(dy'|y,a,C) \nonumber\\
=&\inf_{a\in A}\sup_{C\in\cC^\alpha_t(y)}\left(\int_{E_Y}V_{t+1}(y')Q_t(dy'|y,a,C)-\rho_t(C)\right)=V_t(y).
\end{align}

On the other hand, for any $\varepsilon>0$, let $\varphi^{t+1,\varepsilon}\in\cA_{t+1}$ be an $\varepsilon$-optimal control starting at time $t+1$.
We obtain that
\begin{align*}
&\sup_{\bQ^{\psi^{t+1}}\in\cQ^{\varphi^{t+1,\varepsilon}}_{y',t+1}}\left(\bE_{\bQ^{\psi^{t+1}}}[\ell(X_T)]-\rho_{t:T-1}((C,\psi^{t+1}))\right)\\
\leq&\inf_{\varphi^{t+1}\in\cA_{t+1}}\sup_{\bQ^{\psi^{t+1}}\in\cQ^{\varphi^{t+1}}_{y',t+1}}\left(\bE_{\bQ^{\psi^{t+1}}}[\ell(X_T)]-\rho_{t:T-1}((C,\psi^{t+1}))\right)+\varepsilon\\
=&\left(V_{t+1}(y')-\lambda_t\rho_t(C)\right)+\varepsilon.
\end{align*}
In what follows, we have
\begin{align*}
&\inf_{\varphi\in\cA_t}\sup_{\bQ^\psi\in\cQ_{y,t}^\varphi}\left(\bE_{\bQ^\psi}[\ell(X_T)]-\rho_{t:T-1}(\psi)\right)\\
\leq&\inf_{a\in A}\sup_{\bQ^{\psi}\in\cQ^{(a,\varphi^{t+1,\varepsilon})}_{y',t+1}}\left(\bE_{\bQ^{\psi}}[\ell(X_T)]-\rho_{t:T-1}((C,\psi^{t+1}))\right)\\
=&\inf_{a\in A}\sup_{C\in\cC^\alpha_t(y)}\int_{E_Y}\sup_{\bQ^{\psi^{t+1}}\in\cQ^{\varphi^{t+1,\varepsilon}}_{y',t+1}}\left(\bE_{\bQ^{\psi^{t+1}}}[\ell(X_T)]-\rho_{t:T-1}((C,\psi^{t+1}))\right)Q_t(dy'|y,a,C)\\
\leq&\inf_{a\in A}\sup_{C\in\cC^\alpha_t(y)}\int_{E_Y}\left(V_{t+1}(y')-\rho_t(C)\right)Q_t(dy'|y,a,C)+\varepsilon\\
=&\inf_{a\in A}\sup_{C\in\cC^\alpha_t(y)}\left(\int_{E_Y}V_{t+1}(y')Q_t(dy'|y,a,C)-\rho_t(C)\right)+\varepsilon\\
=&V_t(y)+\varepsilon.
\end{align*}
Since $\varepsilon$ is arbitrary, as a consequence,
\begin{align}\label{eq:roe}
\inf_{\varphi\in\cA_t}\sup_{\bQ^\psi\in\cQ_{y,t}^\varphi}\left(\bE_{\bQ^\psi}[\ell(X_T)]-\rho_{t:T-1}(\psi)\right)\leq V_t(y).
\end{align}
Combine \eqref{eq:loe} and \eqref{eq:roe}, we conclude that
\begin{align*}
\inf_{\varphi\in\cA_t}\sup_{\bQ^\psi\in\cQ_{y,t}^\varphi}\left(\bE_{\bQ^\psi}[\ell(X_T)]-\rho_{t:T-1}(\psi)\right)=V_t(y).
\end{align*}
\end{proof}

The above theorem shows that the problem~\eqref{eq:ar2} is solved via the Bellman equation~\eqref{eq:bellman1}.
In the next section, we will compare the approach \eqref{eq:ar2} to some different methods that can be used to handle the optimization problem under dependence uncertainty.

\subsection{Comparison to Robust Control Problem via Empirical Distribution}\label{sec:comparison}

In this section, we will compare three different setups that can be used for solving the proposed problem.
One of other approaches that is viable in this setup is, as we mentioned earlier, the multidimensional version of the methodology discussed in \cite{BC2022}.
Briefly speaking, one can build confidence regions in terms of Wasserstein balls centered at the empirical distribution $\widehat{F}_t$, say, denoted as $\cC^{\alpha,e}_t(\widehat{F}_t)$.
Then, formulate the adaptive robust problem that optimizes the expected loss/utility against the worst-case model in $\cC^{\alpha,e}_t(\widehat{F}_t)$.
Obviously, doing so will ignore the available information of the marginals $F^{(i)}_*$, $i\in\cN$.
However, for large value of $T$, such method will have a decent performance as the corresponding value function will get quite closed to the true one when $t$ approaches $T$ (cf. \cite{BC2022}).
The associated Bellman equation is given as follows
\begin{align}\label{eq:ar_emp}
V^e_t(y')=\inf_{a\in A}\sup_{F\in\cC^{\alpha,e}_t(y')}\bE_F[V^e_{t+1}(\mathbf{G}(t,y',a,Z_{t+1}))], \quad t\in\cT',
\end{align}
where $y'$ is the state of the process $Y'_t=(X_t,\widehat{F}_t)$, $t\in\cT'$.
The third method is somewhat similar to the one we just described but utilizes the information of the known marginals.
One can impose essentially the same penalty function as in this work and force the worst-case distribution to have correct marginals.
To this end, we define the following one-step penalty function
\begin{align*}
\rho^m_t(F) = \sum_{i\in\cN}\sup_{f_i\in C_b(\bS^{(i)})}\left(\bE_{F^{(i)}}[f_i]-\bE_{F^{(i)}_*}[f_i]\right),
\end{align*}
for any $F\in\cC^{\alpha,m}_t(\widehat{F}_t)$, which is the Wasserstein ball centered at $\widehat{F}_t$ that might have a different radius compared to $C^{\alpha,e}_t(\widehat{F}_t)$.
Also, the above $\bS^{(i)}$ is the support of $Z^{(i)}_{t+1}$, $i\in\cN$. 
The corresponding Bellman equation is
\begin{align}\label{eq:ar_emp_m}
V^m_t(y')=\inf_{a\in A}\sup_{F\in\cC^{\alpha,m}_t(y')}\left(\bE_F[V^m_{t+1}(\mathbf{G}(t,y',a,Z_{t+1}))]-\rho_t(F)\right), \quad t\in\cT',
\end{align}
where $y'$ is defined the same way as in \eqref{eq:ar_emp}.

Now the question is that which framework is the best in what situation, and can the methods with imposed penalty functions take advantage of the available information of the marginals and outperform the approach that ignores it?
First off, note that one can always take $C^{\alpha,m}_t=C^{\alpha,e}_t$ backed by the concentration result.
Then, due to the presence of the penalty function $\rho_t(F)$, framework \eqref{eq:ar_emp_m} is essentially searching for the worst-case distribution in a smaller set compared to \eqref{eq:ar_emp}.
Therefore, we have
\begin{align*}
V_t^m(y')\leq V_t^e(y').
\end{align*}
In other words, \eqref{eq:ar_emp_m} yields smaller loss which is meant to be minimized.
In this sense, the solution corresponding to \eqref{eq:ar_emp_m} is better than the one corresponding to \eqref{eq:ar_emp}.

The comparison between \eqref{eq:ar2} and \eqref{eq:ar_emp_m} is much more complicated, and it turns out that the preference depends on additional properties of $F^*$.
To this end, we first discuss the relationship between the uncertainty sets $\cC^{\alpha}_t(y)$ and $\cC^{\alpha,e}_t(y')$ where $y=(x,\widehat{C})$ and $y'=(x,\widehat{F})$ such that $\widehat{C}=\widehat{F}\circ F^{-1}_*$.
This means that both methods use the exact same information at every time step $t\in\cT'$.
In the first case, when
there exists a constant $B>1$ such that
\begin{align}\label{eq:lip1}
\left\|F_*\circ z_1-F_*\circ z_2\right\|\leq B\|z_1-z_2\|,
\end{align}
where
\begin{align*}
F_*\circ z=\left(F^{(1)}_*(z^{(1)}),\ldots,F^{(n)}_*(z^{(n)})\right).
\end{align*}
Hence, for any $f_1, f_2$ on $[0,1]^n$, such that
$$
|f_1(u_1)-f_2(u_2)|\leq\|u_1-u_2\|^p, \quad u_1,u_2\in [0,1]^n,
$$
we get
\begin{align*}
|f_1(F_*\circ z_1)-f_2(F_*\circ z_2)|\leq\|F_*\circ z_1-F_*\circ z_2\|^p\leq B^p\|z_1-z_2\|^p, \quad z_1,z_2\in\bR^n.
\end{align*}
Following such obervation, we have for any $C$
\begin{align}\label{eq:radius_comp}
d_{W,p}^p(\widehat{C},C)&=\sup_{f_1,f_2}\left(\int_{[0,1]^n}f_1(u)d\widehat{C}(u)-\int_{[0,1]^n}f_2(u)dC(u)\right) \nonumber\\
&=\sup_{f_1,f_2}\left(\int_{\bR^n}f_1 (F_*\circ z)d\widehat{C}(F_*\circ z)-\int_{\bR^n}f_2(F_*\circ z)d C(F_*\circ z)\right)\nonumber\\
&\leq B^p\sup_{g_1,g_2}\left(\int_{\bR^n}g_1(z)d\widehat{C}(F_*\circ z)-\int_{\bR^n}g_2(z)d C(F_*\circ z)\right)\nonumber\\
&=B^p d^p_{W,p}(\widehat{F},C\circ F_*),
\end{align}
where $|g_1(z_1)-g_2(z_2)|\leq\|z_1-z_2\|^p$ for any $z_1,z_2\in\bR^n$.
Thus, for the radius $r^m(\alpha,t_0,t)$ of $\cC^{\alpha,m}_t(y')$, it suffices to take $r^m(\alpha,t_0,t)=r(\alpha,t_0,t)/B$.
Note that, in theory one should choose $r^e(\alpha,t_0,t)$, which is the radius of $C^{\alpha,e}_t$, as $r^e(\alpha,t_0,t)=r(\alpha,t_0,t)$.
Hence, in this case, \eqref{eq:ar_emp_m} will perform better than \eqref{eq:ar_emp} because it has a smaller associated Wasserstein ball.

On the other hand, between \eqref{eq:ar2} and \eqref{eq:ar_emp_m}, in view of \eqref{eq:radius_comp}, we obtain that for any $F\in\cC^{\alpha,m}_t(y')$, there exists a corresponding $C\in\cC^{\alpha}_t(y)$.
Consequently, we can view $\cC^{\alpha,m}_t$ as a ``smaller'' set than $\cC^{\alpha}_t$.
Moreover, we immediately get that
\begin{align*}
V^m_t(y')\leq V_t(y),
\end{align*}
and \eqref{eq:ar_emp_m} outperforms \eqref{eq:ar2}.
One should note that, in this case, it is unclear which one is better between \eqref{eq:ar2} and \eqref{eq:ar_emp_m}.

The other case is that each $\bS^{(i)}$, $i\in\cN$, is a compact subset of $\bR$, and we suppose that there exists a constant $B'>0$ such that
\begin{align}\label{eq:lip2}
\|F_*\circ z'_1-F_*\circ z'_2\|\geq B'\|z'_1-z'_2\|,
\end{align}
for any $z'_1,z'_2\in\prod_{i\in\cN}\bS^{(i)}$.
Under these assumptions, for any $g_1,g_2$ on $\bR^n$, such that
$$
|g_1(z'_1)-g_2(z'_2)|\leq\|z'_1-z'_2\|^p, \quad z'_1,z'_2\in\prod_{i\in\cN}\bS^{(i)},
$$
we obtain
$$
|g_1(F_*^{-1}\circ u_1)-g_2(F_*^{-1}\circ u_2)|\leq\|F_*^{-1}\circ u_1-F_*^{-1}\circ u_2\|^p\leq\frac{1}{B'}\|u_1-u_2\|, \quad u_1,u_2\in[0,1]^n.
$$
Then, for any $F$ whose associated random variable has support $\prod_{i\in\cN}\bS^{(i)}$, and $F=C\circ F_*$ for some copula $C$, we have
\begin{align*}
d^p_{W,p}(\widehat{F},F)=&\sup_{g_1,g_2}\left(\int_{\bR^n}g_1(z)d\widehat{F}(z)-\int_{\bR^n}g_2(z)dF(z)\right)\\
=&\sup_{g_1,g_2}\left(\int_{\prod_{i\in\cN}\bS^{(i)}}g_1(z)d\widehat{F}(z)-\int_{\prod_{i\in\cN}\bS^{(i)}}g_2(z)dF(z)\right)\\
=&\sup_{g_1,g_2}\left(\int_{\prod_{i\in\cN}\bS^{(i)}}g_1(z)d\widehat{C}(F_*\circ z)-\int_{\prod_{i\in\cN}\bS^{(i)}}g_2(z)dC(F_*\circ z)\right)\\
=&\sup_{g_1,g_2}\left(\int_{[0,1]^n}g_1(F_*^{-1}(u)d\widehat{C}(u)-\int_{[0,1]^n}g_2(F_*^{-1}(u)dC(u)\right)\\
\leq&\frac{1}{B'}\sup_{f_1,f_2}\left(\int_{[0,1]^n}f_1(u)d\widehat{C}(u)-\int_{[0,1]^n}f_2(u)d C(u)\right)\\
=&\frac{1}{B'}d^p_{W,p}(\widehat{C},C),
\end{align*}
where $|f_1(u_1)-f_2(u_2)|\leq\|u_1-u_2\|$ for any $u_1,u_2\in[0,1]^n$.
As a result, similarly to the above discussion, we view $\cC^{\alpha}_t$ as a ``smaller'' set than $\cC^{\alpha,m}_t$, and we conclude
\begin{align*}
V_t(y)\leq V^m_t(y')\leq V^e_t(y').
\end{align*}
In other words, the framework proposed in this work will perform the best among all three mentioned approaches.

Finally, we remark that when $\bS^{(i)}$, $i\in\cN$, are compact, and both \eqref{eq:lip1} and \eqref{eq:lip2} hold true, the two
frameworks \eqref{eq:ar2} and \eqref{eq:ar_emp_m} are equivalent. If
neither \eqref{eq:lip1} or \eqref{eq:lip2} is satisfied, it is unclear which one between \eqref{eq:ar2} and \eqref{eq:ar_emp_m} is more preferable.

\section{Uncertain Utility Maximization with Known Marginal Distributions}\label{sec:numerics}

In this section, we will apply the adaptive robust control methodology to the uncertain utility maximization problem where the marginal distributions of the underlying random noise are known.
Note that even though our theory applies to both empirical distribution and perturbed empirical distribution, we will focus only on the case of empirical distribution in this section.
The advantage of using the perturbed empirical distribution is that the resulting $\widehat{C}$ is a copula.
This could be a more sounding approach in some applications, however in this work, it will increase the computation burden which is already heavy.
Hence, we leave the corresponding investigating for potential future studies.

Towards this end, we first discuss the algorithm that we will implement.

\subsection{Numerical Algorithm}\label{sec:algorithm}

From the practical point of view, one major difficulty in solving a robust control problem with Wasserstein uncertainty set such as $\cC^\alpha_t$ is that direct searching through $\cC^\alpha_t(y)$ to find the optimizer $\psi^*_t(y)$ is impossible.
Typically, one uses a duality argument to re-write such problem as a scalar optimization problem.
In \cite{EK18}, the authors provide a rather comprehensive discussion on the tractable reformulations of an optimization problem over the Wasserstein ball.
With some postulation on the loss function, the authors achieve a convex reduction of the worst-case expectation problems.
However, results in \cite{EK18} only apply to the case of optimizing the expected loss while in this work we deal with a different problem as in \eqref{eq:bellman1}
\begin{align}\label{eq:bellman-worst-copula}
\sup_{C\in\cC^{\alpha}_t(y)}\left(\bE_{C\circ F_*}[V_{t+1}(\mathbf{G}(t,y,a,Z_{t+1})]-\rho_t(C)\right).
\end{align}
Due to the non-linearity of the penalty function $\rho_t$, one cannot directly adopt the idea in \cite{EK18} to reformulate problem \eqref{eq:bellman1} for the numerical purpose.

To this end, we recall the paper \cite{GK17} which considers the numerical solution of a 1-period robust optimization problem with Wasserstein uncertainty set centered at the estimated copula which is similar to our worst-case copula optimization problem.
Let $\{z_{-t_0+1},\ldots,z_t\}$ be the historical data of $Z$ up to time $t\in\cT'$, the strong duality result in \cite{GK17} implies that problem \eqref{eq:bellman-worst-copula} can be solved as
\begin{align}\label{eq:tractable}
\inf_{\gamma\in\bR_+,f_i\in C_b(\bR)}\left\{\gamma r(\alpha,t_0,t)^p+\sum_{i=1}^n\int_{\bR}f_i(z)dF_*^{(i)}(z)+\frac{1}{t_0+t}\sum_{j=-t_0+1}^tV^\gamma_{t+1}(\mathbf{G}(t,y,a,z_j)\right\},
\end{align}
where
\begin{align}\label{eq:inter}
V^\gamma_{t+1}(\mathbf{G}(t,y,a,z_j)=\sup_{z\in\bR^n}\left(V_{t+1}(\mathbf{G}(t,y,a,z))-\sum_{i=1}^nf_i(z^{(i)})-\gamma d^p_{F_*}(z,z_j)\right),
\end{align}
and
\begin{align}\label{eq:premetric}
d_{F_*}(\xi,\zeta)=d\left(\left(F_*^{(1)}(\xi^{(1)}),\ldots,F_*^{(n)}(\xi^{(n)})\right),\left(F_*^{(1)}(\zeta^{(1)}),\ldots,F_*^{(n)}(\zeta^{(n)})\right)\right).
\end{align}
\begin{remark}
In general, the premetric $d_{F_*}$ is defined as
\begin{align*}
d_{F_*}(\xi,\zeta)=\liminf_{d(\xi_i,\xi),d(\zeta_i,\zeta)\to0}d\left(\left(F_*^{(1)}(\xi_i^{(1)}),\ldots,F_*^{(n)}(\xi_i^{(n)})\right),\left(F_*^{(1)}(\zeta_i^{(1)}),\ldots,F_*^{(n)}(\zeta_i^{(n)})\right)\right).
\end{align*}
Due to our assumption that $F_*^{(i)}$, $i\in\cN$, are continuous, so $d_{F_*}$ becomes \eqref{eq:premetric}.
\end{remark}

In \cite{GK17}, to use the strong duality result in practice, the authors essentially assume that the range of the random variable $Z_t$, $t\in\cT$, is a finite set.
Such assumption is obviously not true in this work, and hence we will use the following treatment to overcome such obstacle.
Towards this end, denote by $C([0,1])$ the set of all continuous functions on $[0,1]$, we first apply the change of variable to \eqref{eq:tractable} -- \eqref{eq:inter} as follows
\begin{align*}
\inf_{\gamma\in\bR_+,f_i\in C([0,1])}\left\{\gamma r(\alpha,t_0,t)^p+\sum_{i=1}^n\int_0^1f_i\left(F_*^{(i),-1}(u)\right)du+\frac{1}{t_0+t}\sum_{j=-t_0+1}^t V^\gamma_{t+1}\left(\mathbf{G}\left(t,y,a,F_*^{-1}(u_j)\right)\right)\right\},
\end{align*}
where
\begin{align*}
V^{\gamma}_{t+1}\left(\mathbf{G}\left(t,y,a,F^{-1}_*(u_j)\right)\right)=\sup_{u\in[0,1]^n}\left(V_{t+1}\left(\mathbf{G}\left(t,y,a,F_*^{-1}(u)\right)\right)-\sum_{i=1}^nf_i\left(F_*^{(i),-1}(u^{(i)})\right)-\gamma d^p(u,u_j)\right).
\end{align*}
Next, we recall the Weierstrass approximation theorem which indicates that any $f_i\circ F_*^{(i),-1}\in C([0,1])$ can be approximated by polynomials.
In addition, the Bernstein polynomials
$$
\beta_{k,K}(u)=\binom{K}{k}u^k(1-u)^{K-k}, \quad u\in[0,1],\ k=0,\ldots,K,
$$
form a basis for the space of polynomials of degree $K$.
Also, note that for fixed $K$ and any $k=0,\ldots,K$,
$$
\int_{[0,1]}\beta_{k,K}(u)du=\frac{1}{K+1}.
$$
Hence, the above optimization problem is approximated as
\begin{align*}
\inf_{\gamma\in\bR_+,g_{i,k}\in\bR}\left\{\gamma r(\alpha,t_0,t)^p+\sum_{i=1}^n\sum_{k=0}^K\frac{g_{i,k}}{K+1}+\frac{1}{t_0+t}\sum_{j=-t_0+1}^t V^\gamma_{t+1}(\mathbf{G}(t,y,a,F_*^{-1}(u_j)))\right\},
\end{align*}
where
\begin{align*}
V^{\gamma}_{t+1}(\mathbf{G}(t,y,a,F^{-1}_*(u_j)))=\sup_{u\in[0,1]^n}\left(V_{t+1}(\mathbf{G}(t,y,a,F_*^{-1}(u)))-\sum_{i=1}^n\sum_{k=0}^Kg_{i,k}\beta_{k,K}(u^{(i)})-\gamma d^p(u,u_j)\right).
\end{align*}
By including the optimization over feasible controls, the inf-sup problem we need to solve is
\begin{align}\label{eq:bellman2}
\inf_{a\in A,\gamma\in\bR_+,g_{i,k}\in\bR}\left\{\gamma r(\alpha,t_0,t)^p+\sum_{i=1}^n\sum_{k=0}^K\frac{g_{i,k}}{K+1}+\frac{1}{t_0+t}\sum_{j=-t_0+1}^t V^\gamma_{t+1}(\mathbf{G}(t,y,a,F_*^{-1}(u_j)))\right\},
\end{align}
where
\begin{align}\label{eq:constr}
V^{\gamma}_{t+1}(\mathbf{G}(t,y,a,F^{-1}_*(u_j)))=\sup_{u\in[0,1]^n}\left(V_{t+1}(\mathbf{G}(t,y,a,F_*^{-1}(u)))-\sum_{i=1}^n\sum_{k=0}^Kg_{i,k}\beta_{k,K}(u^{(i)})-\gamma d^p(u,u_j)\right).
\end{align}

With regard to designing a numerical algorithm to solve \eqref{eq:bellman2} -- \eqref{eq:constr}, we first note that the inf problem in \eqref{eq:bellman2} and the sup problem in \eqref{eq:constr} have dimensions $m+n*(K+1)+1$ and $n$, respectively.
Hence, we have a non-trivial high dimensional inf-sup optimization problem for which a solver using brute force will be extremely inefficient and virtually impossible.
On top of the inherent high dimension in optimization, it also requires large amount of computation cost to evaluate $\frac{1}{t_0+t}\sum_{j=-t_0+1}^t V^\gamma_{t+1}(\mathbf{G}(t,y,a,F_*^{-1}(u_j)))$ given that $t_0$ is decently big.
In view of these difficulties, we propose to use the stochastic gradient descent ascent (SGDA) method to overcome the obstacles.

To explain the idea of SGDA in our work, for any fixed $t\in\cT'$, $y\in E_Y$, we denote
\begin{align*}
\hat{v}_t(a,\gamma,g,u;\hat{u})=&\gamma (r(\alpha,t_0,t)^p-d^p(u,\hat{u}))+\sum_{i=1}^n\sum_{k=0}^K\left(\frac{g_{i,k}}{K+1}-g_{i,k}\beta_{k,K}(u^{(i)})\right)\\
&+V_{t+1}(\mathbf{G}(t,y,a,F^{-1}(u))),
\end{align*}
where $\hat{u}\in\{u_j, j=-t_0+1,\ldots,t\}$.
Such formulation comes from the considered loop as follows.
Given some $(a(l),\gamma(l),g(l),u(l))$, $l=0,1,\ldots$,
\begin{enumerate}
\item
uniformly simulate $\hat{u}(l)\in\{u_j, j=-t_0+1,\ldots,t\}$;
\item
update $(a(l),\gamma(l),g(l),u(l))$ as
\begin{align*}
a(l+1) &= a(l)-\eta(l)\frac{\partial}{\partial a}\hat{v}_t(a(l),\gamma(l),g(l),u(l);\hat{u}(l)),\\
\gamma(l+1) &= \gamma(l)-\eta(l)\frac{\partial}{\partial \gamma}\hat{v}_t(a(l),\gamma(l),g(l),u(l);\hat{u}(l)),\\
g(l+1) &= g(l)-\eta(l)\frac{\partial}{\partial g}\hat{v}_t(a(l),\gamma(l),g(l),u(l);\hat{u}(l)),\\
u(l+1) &= u(l)+\eta(l)\frac{\partial}{\partial u}\hat{v}_t(a(l),\gamma(l),g(l),u(l);\hat{u}(l)),
\end{align*}
where $\eta(l)$ denote the step size;
\item
Goto 1.
\end{enumerate}
The above loop will stop when there is no improvement of $(a(l),\gamma(l),g(l),u(l))$ anymore, and we denote the corresponding number $l$ of iterations as $l^*$.

To apply SGDA, one needs to have a functional representation of $V_t$, $t=1,\ldots,T-1$, such that the above gradients of $\hat{v}_t$ can be quickly computed.
To construct such functional approximation, we first choose $N$ so-called design points $y_t^j$, $j=1,\ldots,N$, at which we solve for $V_t(y_t^j)$, $j=1,\ldots,N$, and then build a regression model $\widetilde{V}_t$ based on the training data $(y_t^j,V_t(y_t^j))$, $j=1,\ldots,N$.
For the choice of the regression model, we will use the Gaussian process (GP) with kernel function $\mathbf{k}(\cdot,\cdot)$ of the Matern-3/2 type.
For detailed discussion of GP, one can refer to \cite{RW06}. We also summarize the basic idea of GP in earlier works (cf. \cite{CL2019}, \cite{BCC2020}, \cite{CM2020}, \cite{BC2022}).

When it comes to regression, one needs to note that the second component $\widehat{C}$ of the state variable $y$ is in general a distribution and hence infinitely dimensional, which cannot be regressed against numerically.
To such end, we will approximate each $\widehat{C}$ with its first $m$ moments of marginals and covariance among marginals.
By doing so, the dimension of the state variable $y$ is reduced to $mn+\binom{n}{2}+1$ and we denote the approximating vector by $\tilde{y}$.

The rationale behind choosing GP with the Matern-3/2 kernel is twofold.
On one hand, it is a regression model sophisticated enough to capture the complicated structure of $V_t$ with minimal assumptions, as Matern-3/2 kernel is used for regressing the function that has first order derivative.
On the other hand, the gradients of $\widetilde{V}_t$, and in turn of $\hat{v}_t$, can be analytically computed which is of particular importance to our implementation.
To see this, recall that given $y^j_t$, $j=1,\ldots,N$, the resulting regression model has the following representation
\begin{align}\label{eq:v_tilde}
\widetilde{V}_t(\tilde{y}) = (\mathbf{k}(\tilde{y},\tilde{y}^1_t),\ldots,\mathbf{k}(\tilde{y},\tilde{y}^N_t))[\mathbf{K}+\varepsilon^2\mathbf{I}]^{-1}(V_t(y^1_t),\ldots,V_t(y^N_t))^\top.
\end{align}
The above $\varepsilon^2\mathbf{I}$ with $\mathbf{I}$ being the identity matrix is the perturbation added to the matrix $\mathbf{K}$ to ensure that the latter is invertible.
The matrix $\mathbf{K}$ satisfies that $\mathbf{K}_{ij}=\mathbf{k}(\tilde{y}^i_t,\tilde{y}^j_t)$ and
\begin{align*}
\mathbf{k}(\tilde{y}^1,\tilde{y}^2)=\left(1+\sqrt{3}d_s(\tilde{y}^1,\tilde{y}^2)\right)\exp\left(-\sqrt{3}d_s(\tilde{y}^1,\tilde{y}^2)\right),
\end{align*}
where $d_s$ is the scaled Euclidean distance between $y^1$ and $y^2$ with parameter $s$:
\begin{align*}
  d\left(\left(\frac{\tilde{y}^{1,(1)}}{s^{(1)}},\ldots,\frac{\tilde{y}^{1,(n)}}{s^{(n)}}\right),\left(\frac{\tilde{y}^{2,(1)}}{s^{(1)}},\ldots,\frac{\tilde{y}^{2,(n)}}{s^{(n)}}\right)\right),
\end{align*}
where $d(\cdot,\cdot)$ is the usual Euclidean distance.
The vector $s$ is called the length scale and fitting a GP model means to use the maximum likelihood method to find the optimal value of $s$.
After $\widetilde{V}_t$ is fitted, the product $[\mathbf{K}+\varepsilon^2\mathbf{I}]^{-1}(V_t(y^1_t),\ldots,V_t(y^N_t))^\top$ in \eqref{eq:v_tilde} is a fixed constant vector $(\nu^1_t,\ldots,\nu^N_t)$.
Therefore, \eqref{eq:v_tilde} can be rewritten as $\widetilde{V}_t(\tilde{y})=\sum_{j=1}^N\nu^j_t\mathbf{k}(\tilde{y},\tilde{y}^j_t)$.
As a result, gradient of $\widetilde{V}_t$ is a linear combination of the gradients of $\mathbf{k}(\tilde{y},\tilde{y}^j_t)$ and with the explicit formula of $\mathbf{k}$ being available, its gradient can therefore be computed analytically.
There is no need to stress how important the capability of computing the gradient without any numerical approximation is in a gradient based optimization algorithm.
Our solver for problem \eqref{eq:bellman2} -- \eqref{eq:constr} is summarized as follows.
For $t=T-2,\ldots,0$,
\begin{enumerate}
\item
Assume that $V_{t+1}(y^i_{t+1})$ and $\varphi^{*,\varepsilon}_{t+1}(y^i_{t+1})$, $i=1,\ldots,N$, are computed. Fit the GP models $\widetilde{V}_{t+1}$ and $\widetilde{\varphi}^{*,\varepsilon}_{t+1}$ by using the training data $(y^i_{t+1},V_{t+1}(y^i_{t+1}))$, and $(y^i_{t+1},\varphi^{*,\varepsilon}_{t+1}(y^i_{t+1}))$, $i=1,\ldots,N$, respectively.
\item
Choose $y^i_t\in E_Y$, $i=1,\ldots,N$.
\item
For each $y^i_t$, choose initial guesses $(a(0),\gamma(0),g(0),u(0))$, and use SGDA to compute\\
$(a(l^*),\gamma(l^*),g(l^*),u(l^*))$.
Set $V_t(y^i_t)=\hat{v}_t(a(l^*),\gamma(l^*),g(l^*),u(l^*);\hat{u}(l^*))$, and $\varphi^{*,\varepsilon}_t(y^i_t)=a(l^*)$.
\item
Goto 1: start the next recursion for $t-1$.
\end{enumerate}
To analyze the performance of the computed optimal control $\varphi^{*,\varepsilon}$, we use forward simulation to estimate the mean terminal loss over the out-of-sample paths.
For $t=0,\ldots,T-1$,
\begin{enumerate}
  \item
  Draw $N'>0$ i.i.d. $Z^1_{t+},\ldots,Z^{N'}_{t+1}$ from the true distribution corresponding to $F_*$.
  \item
  All paths start from the initial $y_0$. The state along path $i$ is updated according to\\
  $\mathbf{G}(t,y^i_t,\widetilde{\varphi}^{*,\varepsilon}_t(y^i_t),Z^i_{t+1})$, $i=1,\ldots,N'$.
  \item
  Obtain the terminal state $X^i_T$, $i=1,\ldots,N'$, and compute
  $$
  V^c:=\frac{1}{N'}\sum_{i=1}^{N'}\ell(X^i_T)
  $$
  as the estimated mean terminal loss.
\end{enumerate}
In Section~\ref{sec:comparison}, we know that there is no complete theoretical result of the comparison between our adaptive robust method based on copula and the approach of using empirical distribution without marginal information.
Hence, we will focus on such comparison when presenting our numerical results.

\begin{figure}[H]
\centering
\begin{tabular}{cc}
\includegraphics[height=1.9in]{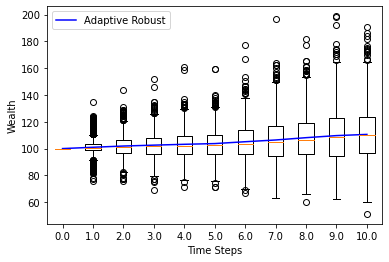} &
\includegraphics[height=1.9in]{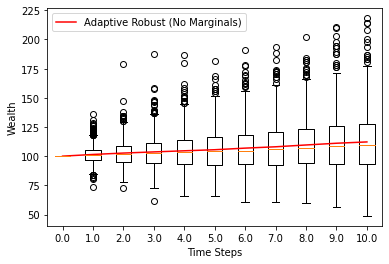} \\
\includegraphics[height=1.9in]{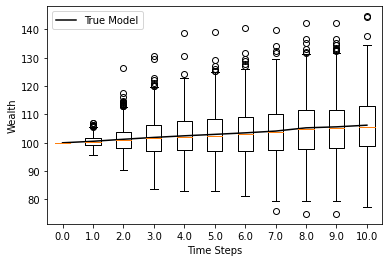} &
\includegraphics[height=1.9in]{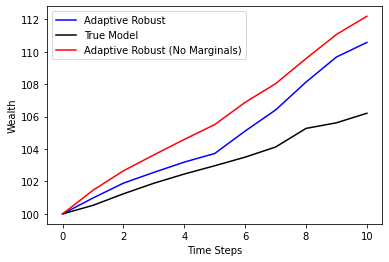} \\
\end{tabular}
\caption{Evolution of the wealth along out-of-sample paths. Upper left: box-plot of the wealth generated by $\varphi^{*,\varepsilon}$; Upper right: box-plot of the wealth generated by $\varphi^e$; Bottom left: box-plot of the wealth generated by $\varphi^{\text{tr}}$; Bottom right: comparison of mean wealth.}
\label{fig:wealth}
\end{figure}

\subsection{Numerical Results}

In this section, we apply the algorithm described above to the uncertain utility maximization problem.
We take $d=2$ and consider two different stocks where $Z_{t+1}$ is the vector of log-returns from time $t$ to $t+1$, which follows a multi-dimensional normal distribution.
Correspondingly, let $X_t$ be the wealth at time $t$ with the dynamics
$$
X_t=X_{t-1}\left((1-\varphi_t^{(1)}-\varphi_t^{(2)})(1+r)+\varphi_t^{(1)}e^{Z_t^{(1)}}+\varphi_t^{(2)}e^{Z_t^{(2)}}\right).
$$
The above $\varphi_t^{(i)}$, $i=1,2$, is the proportion of wealth invested in stock $i$.
The constant $r$ and $Z^{(i)}_t$, $i=1,2$, are the interest rate and log-return of stock $i$, respectively.
In addition, we choose the loss function $\ell$ such that $-\ell=U$ which is the exponential utility function.
Then, we are effectively facing a utility maximization problem involving two different stocks under model uncertainty.

We compare three different types of strategies: 1. adaptive robust control method with known marginals; 2. adaptive robust control method based on empirical distribution without using marginal information; 3. optimal control without uncertainty.
We denote by $\varphi^{*,\varepsilon}$, $\varphi^e$, and $\varphi^{\text{tr}}$ the corresponding strategies, respectively.
We also use $V^c$, $V^e$, and $V^{\text{tr}}$ to denote the corresponding estimated expected utility, respectively.
In Section~\ref{sec:comparison}, we find that it is unclear which of $\varphi^{*,\varepsilon}$ and $\varphi^e$ will perform better when the marginal CDFs are Lipschitz continuous.
Hence, we will compare these two strategies for Lipschitz marginals by using numerical results presented in the sequel.

To proceed, we take the following values for our parameters: $t_0=400$, $T=10$, $N=N'=1000$, $\alpha=0.1$, $p=2$, and $m=2$.
Regarding the true distribution $F_*$, it is bivariate normal with mean $\mu=(0.09,0.13)$ and covariance matrix $\Sigma=\bigl( \begin{smallmatrix}0.25^2 & 0.85\cdot0.25\cdot0.4\\ 0.85\cdot0.25\cdot0.4 & 0.4^2\end{smallmatrix}\bigr)$ where all numbers are annualized.
Finally, the annual interest rate is $r=0.02$.
With $T=10$, it means that we consider 10 trading periods in a year.

\begin{table}[H]
\centering
\begin{tabular}{c|ccc}
\hline
 \multicolumn{1}{c}{ }& \multicolumn{1}{||c}{AR} & \multicolumn{1}{|c}{AR (No Marginals)} & \multicolumn{1}{|c}{TR} \\
  \hline
   $V$  & \multicolumn{1}{||c}{\ \ 19.8726\ \ } & \multicolumn{1}{|c}{19.8659} & \multicolumn{1}{|c}{\ \ 19.8857\ \ }\\
 $\text{var}(X_T)$   & \multicolumn{1}{||c}{441.7334} & \multicolumn{1}{|c}{628.7182} & \multicolumn{1}{|c}{118.3159}\\
 $q_{0.30}(X_T)$  & \multicolumn{1}{||c}{92.4007} & \multicolumn{1}{|c}{90.4748} & \multicolumn{1}{|c}{96.3813}\\
 $q_{0.90}(X_T)$  & \multicolumn{1}{||c}{136.7146} & \multicolumn{1}{|c}{144.7677} & \multicolumn{1}{|c}{121.2001}\\
 $\text{max}(X_T)$  & \multicolumn{1}{||c}{190.4936} & \multicolumn{1}{|c}{190.4936} & \multicolumn{1}{|c}{144.6299}\\
 $\text{min}(X_T)$  & \multicolumn{1}{||c}{51.2641} & \multicolumn{1}{|c}{48.7990} & \multicolumn{1}{|c}{77.4599}\\
 \hline
\end{tabular}
\bigskip
\caption{Expected utility, variance, 30\%-quantile, 90\%-quantile, maximum, and minimum of the out-of-sample terminal wealth. AR: adaptive robust; TR: no uncertainty.}
\label{table:table1}
\end{table}

Table~\ref{table:table1} shows that, in terms of the expected utility, adaptive robust using marginals performs better than the approach without using information of marginals, but worse than the case of knowing the true model.
Similar comparison results hold true if we look at the 30\% quantile and the minimum value of the terminal wealth.  
On the contrary, $\varphi^e$ produces the highest variance, 90\% quantile, and maximum of the terminal wealth, while $\varphi^{\text{tr}}$ generates the lowest values in these aspects.
These observations can be obtained by viewing the box plots in Figure~\ref{fig:wealth}.
In addition, adaptive robust without information of marginals has the highest mean wealth while knowing the true model gives the lowest value.

As we can see, when the investor knows more about the true model, the resulting terminal wealth excels in the risk management aspects such as variance, 30\% quantile, and minimum value.
Knowing less information of the true model will produce higher mean, 90\% quantile, and maximum of the wealth which signal a strategy that is relatively more aggressive and profit seeking.
By looking at the expected utility which takes into account both profit seeking and risk aversion, we notice that the more information we know about the model the higher mean utility we obtain.
With such numerical results, we conclude that, in general, knowing more about true model leads to a more balanced and less risky optimal strategy.

\bibliographystyle{alpha}
\bibliography{nonparametric}

\end{document}